\documentclass[11pt, final]{article}
\usepackage{a4}
\usepackage{amsmath}%
\usepackage{amstext}%
\usepackage{amssymb}%
\usepackage{showkeys}%
\usepackage{epsfig}%
\usepackage{cite}
\usepackage{tikz}
\usepackage{subfig}
\usepackage{caption}
\usepackage{multirow}
\usepackage{booktabs}
\usepackage{floatrow}

\usepackage{listings}
\newtheorem{theorem}{Theorem}

\newtheorem{axiom}{Axiom}

\newtheorem{corollary}[theorem]{Corollary}

\newtheorem{definition}[axiom]{Definition}

\newtheorem{lemma}[theorem]{Lemma}

\newtheorem{rem}[theorem]{Remark}
\newenvironment{remark}{\rem\rm}{\endrem}

\newcounter{unnumber}

\newtheorem{prf}[unnumber]{Proof.}
\newenvironment{proof}{\prf\rm}{\hfill{$\blacksquare$}\endprf}
\newcommand{\R}{\mathbb{R}}%
\newcommand{\N}{\mathbb{N}}%
\newcommand{\ol}{\overline}%
\newcommand{\ul}{\underline}%

\renewcommand{\>}{\right\rangle}
\newcommand{\<}{\left\langle}

\DeclareMathOperator*\gr{Gr}%
\DeclareMathOperator*\id{Id}%
\DeclareMathOperator*\prox{prox}%
\DeclareMathOperator*\argmin{argmin}

\DeclareMathOperator*\zer{zer}
\DeclareMathOperator*\loc{loc}
\DeclareMathOperator*\fix{Fix}

\title{A dynamical system associated with the fixed points set of a nonexpansive operator}

\author{Radu Ioan Bo\c{t} \thanks{University of Vienna, Faculty of Mathematics, Oskar-Morgenstern-Platz 1, A-1090 Vienna, Austria,
email: radu.bot@univie.ac.at. Research partially supported by DFG (German Research Foundation), project BO 2516/4-1.} \and
Ern\"{o} Robert Csetnek \thanks {University of Vienna, Faculty of Mathematics, Oskar-Morgenstern-Platz 1, A-1090 Vienna, Austria,
email: ernoe.robert.csetnek@univie.ac.at. Research supported by DFG (German Research Foundation), project BO 2516/4-1.}}

\begin{document}
\maketitle

\noindent \textbf{Abstract.} We study the existence and uniqueness of (locally) absolutely continuous trajectories of 
a dynamical system governed by a nonexpansive operator. The weak convergence of the orbits to a fixed point of the 
operator is investigated by relying on Lyapunov analysis. We show also an order of convergence of $o\left(\frac{1}{\sqrt{t}}\right)$ for the fixed point residual 
of the trajectory of the dynamical system.  We apply the results to dynamical systems associated with the problem of finding the zeros of the sum 
of a maximally monotone operator and a cocoercive one. Several dynamical systems from the literature turn out to be particular instances of this general approach. 
\vspace{1ex}

 \vspace{1ex}


\section{Introduction and preliminaries}\label{sec-intr}

Having their origins in the nowadays standard works of Br\'{e}zis, Baillon and Bruck (see \cite{brezis, baillon-brezis1976, bruck}), 
differential inclusions and continuous dynamical systems governed by maximal monotone operators still play an important role 
in optimization and differential equations. While usually the existence and uniqueness of such trajectories is guaranteed 
in the framework of the Cauchy-Lipschitz theorem, their (ergodic) convergence to the set of zeros of the involved maximally monotone 
operators (which in case of the convex subdifferential of a convex function coincides with the set of its minima) relies on Lyapunov analysis. 

In this paper we turn our attention to dynamical systems formulated via resolvents of maximal monotone operators, being motivated by several papers on this subject, like 
\cite{bolte-2003, abbas-att-arx14, att-sv2011, abbas-att-sv, antipin}. In \cite{bolte-2003}, Bolte 
studied the convergence of the trajectories of the following dynamical system
\begin{equation}\label{syst-bolte}\left\{
\begin{array}{ll}
\dot x(t)+x(t)=P_C\big(x(t)-\mu\nabla\phi(x(t))\big)\\
x(0)=x_0.
\end{array}\right.\end{equation}
where $\phi:{\cal H}\rightarrow\R$ is a convex $C^1$ function defined on a real Hilbert space ${\cal H}$, $C$ is a nonempty, closed and convex subset of 
${\cal H}$, $x_0\in {\cal H}$, $\mu>0$ and $P_C$ denotes the projection operator on the set $C$. In this context it is shown that the trajectory of \eqref{syst-bolte} converges weakly 
to a minimizer of the optimization problem  
\begin{equation}\label{opt-bolte} 
\inf_{x\in C}\phi(x),                                                             
\end{equation}
provided the latter is solvable.  We refer also to \cite{antipin} for further statements and results concerning \eqref{syst-bolte}. 

The following generalization of the dynamical system \eqref{syst-bolte} has been recently considered by Abbas and Attouch in \cite[Section 4.2]{abbas-att-arx14}: 
\begin{equation}\label{syst-abb-att}\left\{
\begin{array}{ll}
\dot x(t)+x(t)=\prox_{\mu\Phi}\big(x(t)-\mu B(x(t))\big)\\
x(0)=x_0,
\end{array}\right.\end{equation}
where $\Phi:{\cal H}\rightarrow\R\cup\{+\infty\}$ is a proper, convex and lower semicontinuous function defined on a real Hilbert space ${\cal H}$, 
$B:{\cal H}\rightarrow {\cal H}$ is a cocoercive operator, $x_0\in {\cal H}$, $\mu >0$ and $\prox_{\mu\Phi}:{\cal H}\rightarrow {\cal H}$,
\begin{equation}\label{prox-def}\prox\nolimits_{\mu \Phi}(x)=\argmin_{y\in {\cal H}}\left \{\Phi(y)+\frac{1}{2\mu}\|y-x\|^2\right\},
\end{equation}
denotes the proximal point operator of $\Phi$.

According to \cite{abbas-att-arx14}, in case $\zer(\partial \Phi+B)\neq\emptyset$, the weak convergence of the orbit $x$ of \eqref{syst-abb-att} is 
ensured by choosing the step-size $\mu$ in a suitable domain bounded by the parameter of cocoercivity of the operator $B$ (notice that 
$\partial\Phi$ denotes the convex subdifferential of $\Phi$). 

Let us mention that the time discretization of the dynamical system \eqref{syst-abb-att} leads to the classical 
forward-backward algorithm, a scheme which iteratively generates a sequence that weakly converges to a zero of $\partial \Phi+B$, see 
\cite{abbas-att-arx14} and \cite{bauschke-book}. For more on the relations between the continuous and discrete dynamics we refer the 
reader to \cite{peyp-sorin2010}. We also refer to \cite{b-c-h1, b-c-h2, vu} for more insights into the outstanding role played  by the 
discrete forward-backward algorithm in connection to the solving of complexly structured monotone inclusion problems.

The dynamical systems \eqref{syst-bolte} and \eqref{syst-abb-att} are the starting points of our research. It is known, see \cite{bauschke-book}, that the discrete version of the forward-backward 
algorithm and some of its convergence properties follow form a more general iterative scheme, namely the Krasnosel'ski\u{\i}--Mann algorithm, 
which generates a sequence which approaches the set of fixed points of a nonexpansive operator. Let us mention here that the classical Douglas-Rachford algorithm, 
designed for determining the set of zeros of the sum of two set-valued maximally monotone operators (see \cite{bauschke-book}) can be embedded in the framework of 
the Krasnosel'ski\u{\i}--Mann-type algorithm.

In this paper we study a time-continuous dynamical system which involves a nonexpansive operator, see \eqref{dyn-syst-KM}. Firstly, we address 
the existence and uniqueness of (locally) absolutely continuous trajectories of the considered system, which follows by reformulating in 
the framework of Cauchy-Lipschitz problems and by applying a classical result, see \cite{haraux, sontag}.  In the next section we study the convergence of the trajectories to a fixed point of the operator, 
the investigation relying on Lyapunov analysis combined with the continuous version of the celebrated Opial Lemma. We study also the convergence rates of the fixed point residual of the orbits of the dynamical system, 
for which we obtain a speed of convergence of order $o(1/\sqrt{t})$. Further, we propose a generalization of the forward-backward 
continuous version of the dynamical system \eqref{syst-abb-att} by considering instead of the convex subdifferential a maximally monotone operator and a relaxed backward step. A discussion on possible time-discretizations of 
the investigated dynamical systems is also made. In the last section we present a second approach which reduces the study of the dynamical system \eqref{dyn-syst-KM} via time rescaling arguments to the one of 
autonomous systems governed by cocoercive operators and which allows the formulation of convergence statements under weaker assumptions than in the direct approach.

Let us fix a few notations used throughout the paper. Let $\N= \{0,1,2,...\}$ be the set of nonnegative integers. Let ${\cal H}$ be a real Hilbert space with inner product
$\langle\cdot,\cdot\rangle$ and associated norm $\|\cdot\|=\sqrt{\langle \cdot,\cdot\rangle}$.

\section{A dynamical system: existence and uniqueness of global solutions}\label{sec2}

Let $T:{\cal H}\rightarrow {\cal H}$ be a nonexpansive mapping (that is $\|Tx-Ty\|\leq\|x-y\|$ for all $x,y\in{\cal H}$), 
$\lambda:[0,+\infty)\rightarrow [0,1]$ be a Lebesgue measurable function and $x_0\in {\cal H}$. In this paper we are concerned with the following dynamical system: 

\begin{equation}\label{dyn-syst-KM}\left\{
\begin{array}{ll}
\dot x(t)=\lambda(t)\big(T(x(t))-x(t)\big)\\
x(0)=x_0.
\end{array}\right.\end{equation}

The first issue we investigate is the existence of strong solutions for \eqref{dyn-syst-KM}. As in \cite{att-sv2011, abbas-att-sv}, we consider the following definition of an absolutely continuous function.

\begin{definition}\label{abs-cont} \rm (see, for instance, \cite{att-sv2011, abbas-att-sv}) A function $f:[0,b]\rightarrow {\cal H}$ (where $b>0$) is said to be absolutely continuous if one of the 
following equivalent properties holds: 

(i)  there exists an integrable function $g:[0,b]\rightarrow {\cal H}$ such that $$f(t)=f(0)+\int_0^t g(s)ds \ \ \forall t\in[0,b];$$

(ii) $f$ is continuous and its distributional derivative is Lebesgue integrable on $[0,b]$; 

(iii) for every $\varepsilon > 0$, there exists $\eta >0$ such that for any finite family of intervals $I_k=(a_k,b_k)$ we have the implication:
$$\left(I_k\cap I_j=\emptyset \mbox{ and }\sum_k|b_k-a_k| < \eta\right)\Longrightarrow \sum_k\|f(b_k)-f(a_k)\| < \varepsilon.$$
\end{definition}

\begin{remark}\label{rem-abs-cont}\rm (a) It follows from the definition that an absolutely continuous function is differentiable almost 
everywhere, its derivative coincides with its distributional derivative almost everywhere and one can recover the function from its derivative $f'=g$ 
by the integration formula (i). 

(b) If $f:[0,b]\rightarrow {\cal H}$ (where $b>0$) is absolutely continuous and $B:{\cal H}\rightarrow {\cal H}$ is $L$-Lipschitz continuous
(where $L\geq 0$), then the function $h=B\circ f$ is absolutely continuous. This can be easily verified by considering the characterization in
Definition \ref{abs-cont}(iii). Moreover, $h$ is almost everywhere differentiable and the inequality $\|h'(\cdot)\|\leq L\|f'(\cdot)\|$ holds almost everywhere.   
\end{remark}

\begin{definition}\label{str-sol}\rm We say that $x:[0,+\infty)\rightarrow {\cal H}$ is a strong global solution of \eqref{dyn-syst-KM} if the 
following properties are satisfied: 

(i) $x:[0,+\infty)\rightarrow {\cal H}$ is absolutely continuous on each interval $[0,b]$, $0<b<+\infty$; 

(ii) $\dot x(t)=\lambda(t)\big(T(x(t))-x(t)\big)$ for almost all $t\in[0,+\infty)$;

(iii) $x(0)=x_0$.
\end{definition}

In what follows we verify the existence and uniqueness of strong global solutions of \eqref{dyn-syst-KM}. To this end we use the Cauchy-Lipschitz theorem for absolutely continues trajectories (see for example 
\cite[Proposition 6.2.1]{haraux}, \cite[Theorem 54]{sontag}). 

It is immediate that the system \eqref{dyn-syst-KM} can be written as \begin{equation}\label{existence}\left\{
\begin{array}{ll}
\dot x(t)=f(t,x(t))\\
x(0)=x_0,
\end{array}\right.\end{equation}

where $f:[0,+\infty)\times {\cal H}\rightarrow {\cal H}$ is defined by $f(t,x)=\lambda(t)(Tx-x)$. 

(a) Take arbitrary $x,y\in {\cal H}$. Relying on the nonexpansiveness of $T$, for all $t\geq 0$ we have 
$$\|f(t,x)-f(t,y)\|\leq2\lambda(t)\|x-y\|.$$

Since $\lambda$ is bounded above, one has $2\lambda(\cdot)\in L^1([0,b])$ for any $0<b<+\infty$; 

(b) Take arbitrary $x\in {\cal H}$ and $b>0$. One has 

$$\int_0^b\|f(t,x)\|dt= \|Tx-x\|\int_0^b\lambda(t)dt\leq b\|Tx-x\|,$$
hence $$\forall x\in{\cal H}, \ \forall b>0, \ \ f(\cdot,x)\in L^1([0,b],{\cal H}).$$

By considering the statements proven in (a) and (b), the existence and uniqueness of a strong global solution of
the dynamic system \eqref{dyn-syst-KM} follows. 

\begin{remark}\label{e-u-lambda} From the considerations above one can easily notice that the existence and uniqueness of strong global solutions of \eqref{dyn-syst-KM}  can be guaranteed in the more 
general setting when $T$ is Lipschitz continuous and $\lambda:[0,+\infty)\rightarrow \R$ is a Lebesgue measurable function such that 
$\lambda(\cdot)\in L^1_{\loc}([0,+\infty))$. 
\end{remark}

\section{Convergence of the trajectories}

In this section we investigate the convergence properties of the trajectories of the dynamical system \eqref{dyn-syst-KM}. We show that 
under mild conditions imposed on the function $\lambda$, the orbits converge weakly to a fixed point of the nonexpansive operator, provided 
the set of such points is nonempty. 

In order to achieve this, we need the following preparatory result.

\begin{lemma}\label{fejer-cont2} (\!\!\cite[Lemma 5.2]{abbas-att-sv}) If $1 \leq p < \infty$, $1 \leq r \leq \infty$, $F:[0,+\infty)\rightarrow[0,+\infty)$ is 
locally absolutely continuous, $F\in L^p([0,+\infty))$, $G:[0,+\infty)\rightarrow\R$, $G\in  L^r([0,+\infty))$ and 
for almost all $t$ $$\frac{d}{dt}F(t)\leq G(t),$$ then $\lim_{t\rightarrow +\infty} F(t)=0$. 
\end{lemma}

The next result which we recall here is the continuous version of the Opial Lemma (see for example \cite[Lemma 5.3]{abbas-att-sv}, \cite[Lemma 1.10]{abbas-att-arx14}). 

\begin{lemma}\label{opial} Let $S \subseteq {\cal H}$ be a nonempty set and $x:[0,+\infty)\rightarrow{\cal H}$ a given map. Assume that 

(i) for every $z\in S$, $\lim_{t\rightarrow+\infty}\|x(t)-z\|$ exists; 

(ii) every weak sequential cluster point of the map $x$ belongs to $S$. 

\noindent Then there exists $x_{\infty}\in S$ such that $w-\lim_{t\rightarrow+\infty}x(t)=x_{\infty}$. 
\end{lemma}

The following result, which is a consequence of the demiclosedness principle (see \cite[Theorem 4.17]{bauschke-book}), 
will be used in the proof of Theorem \ref{conv-KM}. which is the main theorem of this paper. 

\begin{lemma}\label{demi}(\!\!\cite[Corollary 4.18]{bauschke-book}) Let $T:{\cal H}\rightarrow {\cal H}$ be nonexpansive and
let $(x_n)_{n\in\N}$ be a sequence in ${\cal H}$ and $x\in {\cal H}$ such that $w-\lim_{n\rightarrow +\infty} x_n=x$ and 
$(Tx_n-x_n)_{n\in\N}$  converges strongly to $0$ (as $n\rightarrow+\infty$). Then $x\in\fix T$.
\end{lemma}

The following identity will be used  
several times in the paper (see for example \cite[Corollary 2.14]{bauschke-book}):
\begin{equation}\label{id-hilb} \|\alpha x+(1-\alpha)y\|^2+\alpha(1-\alpha)\|x-y\|^2=\alpha\|x\|^2+(1-\alpha)\|y\|^2 \ \forall \alpha\in\R
 \ \forall (x,y)\in{\cal H}\times{\cal H}.\end{equation}

\begin{theorem}\label{conv-KM} Let $T:{\cal H}\rightarrow {\cal H}$ be a nonexpansive mapping such that $\fix T\neq\emptyset$, 
$\lambda:[0,+\infty)\rightarrow [0,1]$ a Lebesgue measurable function and 
$x_0\in {\cal H}$. Suppose that one of the following conditions is fulfilled: 
$$\int_0^{+\infty}\lambda(t)(1-\lambda(t))dt=+\infty \ \mbox{or} \ \inf_{t\geq 0}\lambda(t)>0.$$ 
Let  $x:[0,+\infty)\rightarrow{\cal H}$ be the unique strong global solution of \eqref{dyn-syst-KM}. Then the following statements are true: 

(i) the trajectory $x$ is bounded and $\int_0^{+\infty}\|\dot x(t)\|^2dt<+\infty$;  

(ii) $\lim_{t\rightarrow+\infty}(T(x(t))-x(t))=0$; 

(iii) $\lim_{t\rightarrow+\infty}\dot x(t)=0$;

(iv) $x(t)$ converges weakly to a point in $\fix T$, as $t\rightarrow+\infty$.
\end{theorem}

\begin{proof} We rely on Lyapunov analysis combined with the Opial Lemma. We take an arbitrary $y\in\fix T$ and give an estimation for 
$\frac{d}{dt}\|x(t)-y\|^2$. Take an arbitrary $t\geq 0$. By \eqref{id-hilb}, the fact that $y\in\fix T$ and the nonexpansiveness of $T$ we obtain: 
\begin{align*}
\frac{d}{dt}\|x(t)-y\|^2   = & 2\<\dot x(t),x(t)-y\>=\|\dot x(t)+x(t)-y\|^2-\|x(t)-y\|^2-\|\dot x(t)\|^2\\
                           = & \|\lambda(t)(T(x(t))-y)+(1-\lambda(t))(x(t))-y)\|^2-\|x(t)-y\|^2-\|\dot x(t)\|^2\\
                          = & \lambda(t)\|T(x(t))-y\|^2+(1-\lambda(t))\|x(t)-y\|^2 \\
                            & - \lambda(t)(1-\lambda(t))\|T(x(t)-x(t))\|^2 -\|x(t)-y\|^2-\|\dot x(t)\|^2\\
                         \leq & -\lambda(t)(1-\lambda(t))\|T(x(t)-x(t))\|^2-\|\dot x(t)\|^2.
\end{align*}

Hence for all $t\geq 0$ we have that 
\begin{equation}\label{ineq-x-dtx}\frac{d}{dt}\|x(t)-y\|^2+\lambda(t)(1-\lambda(t))\|T(x(t)-x(t))\|^2+\|\dot x(t)\|^2\leq 0.\end{equation}

Since $\lambda(t)\in[0,1]$ for all $t\geq 0$, from \eqref{ineq-x-dtx} it follows that $t\mapsto\|x(t)-y\|$ is decreasing, 
hence $\lim_{t\rightarrow+\infty}\|x(t)-y\|$ exists. 
From here we obtain the boundedness of the trajectory and by integrating \eqref{ineq-x-dtx} we deduce also that $\int_0^{+\infty}\|\dot x(t)\|^2dt<+\infty$ and 
\begin{equation}\label{int-l-T}\int_0^{+\infty}\lambda(t)(1-\lambda(t))\|T(x(t))-x(t)\|^2dt<+\infty,\end{equation}
thus (i) holds. Since $y\in\fix T$ has been chosen arbitrary, the first assumption in the continuous version of Opial Lemma is fulfilled. 

We show in the following that $\lim_{t\rightarrow+\infty}(T(x(t))-x(t))$ exists and it is a real number. This is immediate if we show that 
the function $t\mapsto\frac{1}{2}\|T(x(t))-x(t)\|^2$ is decreasing. According to Remark \ref{rem-abs-cont}(b), the function 
$t\mapsto T(x(t))$ is almost everywhere differentiable and $\|\frac{d}{dt}T(x(t))\|\leq \|\dot x(t)\|$ holds for almost all $t\geq 0$. 
Moreover, by the first equation of \eqref{dyn-syst-KM} we have  
\begin{align*}
\frac{d}{dt}\left(\frac{1}{2}\|T(x(t))-x(t)\|^2\right) = & \<\frac{d}{dt}T(x(t))-\dot x(t),T(x(t))-x(t)\>\\
                                                       = & -\<\dot x(t),T(x(t))-x(t)\>+\<\frac{d}{dt}T(x(t)),T(x(t))-x(t)\>\\
                                                       = & -\lambda(t)\|T(x(t))-x(t)\|^2+\<\frac{d}{dt}T(x(t)),T(x(t))-x(t)\>\\
                                                     \leq & -\lambda(t)\|T(x(t))-x(t)\|^2+\|\dot x(t)\| \cdot \|T(x(t))-x(t)\|=0,
\end{align*}
hence $\lim_{t\rightarrow + \infty}(T(x(t))-x(t))$ exists and is a real number.

(a) Firstly, let us assume that $\int_0^{+\infty}\lambda(t)(1-\lambda(t))dt=+\infty$. This immediately implies by \eqref{int-l-T} 
that $\lim_{t\rightarrow+\infty}(T(x(t))-x(t))=0$, thus (ii) holds. Taking into account that $\lambda$ is bounded, 
from \eqref{dyn-syst-KM} and (ii) we deduce (iii). For the last property of the theorem we need to verify the second assumption  
of the Opial Lemma. Let $\ol x\in {\cal H}$ be a weak sequential cluster point of $x$, that is, there exists a sequence $t_n\rightarrow+\infty$ 
(as $n\rightarrow+\infty$) such that $(x(t_n))_{n\in\N}$ converges weakly to $\ol x$. Applying Lemma \ref{demi} and (ii) we obtain $\ol x\in\fix T$
and the conclusion follows. 

(b) We suppose now that $\inf_{t\geq 0}\lambda(t)>0$. From the first relation of \eqref{dyn-syst-KM} and (i) we easily deduce 
that $Tx-x\in L^2([0,+\infty),{\cal H})$, hence the function $t\mapsto\frac{1}{2}\|T(x(t))-x(t)\|^2$ belongs to $L^1([0,+\infty))$. Since 
$\frac{d}{dt}\left(\frac{1}{2}\|T(x(t))-x(t)\|^2\right)\leq 0$ for almost all $t\geq 0$, we obtain by applying Lemma \ref{fejer-cont2} that 
$\lim_{t\rightarrow+\infty}\|T(x(t))-x(t)\|^2=0$, thus (ii) holds. The rest of the proof can be done in the lines of case (a) considered above.  
\end{proof}

\begin{remark}\label{lambda} Notice that the function $\lambda_1(t)=\frac{1}{t+1}$, for all $t\geq 0$, verifies the condition 
$\int_0^{+\infty}\lambda_1(t)(1-\lambda_1(t))dt=+\infty$, while $\inf_{t\geq 0}\lambda_1(t)>0$ is not fulfilled. On the other hand, the function 
$\lambda_2(t)=1$, for all $t\geq 0$, verifies the condition $\inf_{t\geq 0}\lambda_2(t)>0$, while $\int_0^{+\infty}\lambda_2(t)(1-\lambda_2(t))dt=+\infty$ fails. 
This shows that the two assumptions on $\lambda$ under which the conclusions of Theorem \eqref{conv-KM} are valid are independent.
\end{remark}

\begin{remark}\label{contdisc} The explicit discretization of \eqref{dyn-syst-KM} with respect to the time variable $t$, with step size $h_n>0$, 
yields for an initial point $x_0$ the following iterative scheme: 
$$x_{n+1}=x_n+h_n\lambda_n(Tx_n-x_n) \ \forall n \geq 0.$$
By taking $h_n=1$ this becomes \begin{equation}\label{KM-discrete}x_{n+1}=x_n+\lambda_n(Tx_n-x_n) \ \forall n \geq 0,\end{equation}
which is the classical Krasnosel'ski\u{\i}--Mann algorithm for finding the set of fixed points of the nonexpansive operator $T$ 
(see \cite[Theorem 5.14]{bauschke-book}). Let us mention that the convergence of \eqref{KM-discrete} is guaranteed under the condition 
$\sum_{n \in \N} \lambda_n(1-\lambda_n)=+\infty$. Notice that in case $\lambda_n=1$ for all $n \in\N$ and for an initial point $x_0$ different from $0$, 
the convergence of \eqref{KM-discrete} can fail, as it happens for instance for the operator $T=-\id$. 
In contrast to this, as pointed out in Theorem \ref{conv-KM}, the dynamical system \eqref{dyn-syst-KM} has a strong global solution and the convergence 
of the trajectory is guaranteed also in case $\lambda(t)=1$ for all $t\geq 0$. 
\end{remark}

An immediate consequence of Theorem \ref{conv-KM} is the following corollary, where we consider dynamical systems involving averaged operators. 
Let $\alpha\in(0,1)$ be fixed. We say that $R:{\cal H}\rightarrow{\cal H}$ is {\it $\alpha$-averaged} if there exists a nonexpansive operator 
$T:{\cal H}\rightarrow{\cal H}$ such that $R=(1-\alpha)\id+\alpha T$. For $\alpha=\frac{1}{2}$ we obtain as an important representative of this class the
firmly nonexpansive operators. For properties and other insides concerning these families  of operators we refer to \cite{bauschke-book}. 

\begin{corollary}\label{conv-KM-av} Let $\alpha\in (0,1)$, $R:{\cal H}\rightarrow {\cal H}$ be $\alpha$-averaged such that $\fix R\neq\emptyset$, 
$\lambda:[0,+\infty)\rightarrow [0,1/\alpha]$ a Lebesgue measurable function and 
$x_0\in {\cal H}$. Suppose that one of the following conditions is fulfilled: 
$$\int_0^{+\infty}\lambda(t)(1-\alpha\lambda(t))dt=+\infty \ \mbox{or} \ \inf_{t\geq 0}\lambda(t)>0.$$ Let 
$x:[0,+\infty)\rightarrow{\cal H}$ be the unique strong global solution of the dynamical system 
\begin{equation}\label{dyn-syst-KM-av}\left\{
\begin{array}{ll}
\dot x(t)=\lambda(t)\big(R(x(t))-x(t)\big)\\
x(0)=x_0.
\end{array}\right.\end{equation} Then the following statements are true: 

(i) the trajectory $x$ is bounded and $\int_0^{+\infty}\|\dot x(t)\|^2dt<+\infty$;  

(ii) $\lim_{t\rightarrow+\infty}(R(x(t))-x(t))=0$; 

(iii) $\lim_{t\rightarrow+\infty}\dot x(t)=0$;

(iv) $x(t)$ converges weakly to a point in $\fix R$, as $t\rightarrow+\infty$.
\end{corollary}

\begin{proof} Since $R$ is $\alpha$-averaged, there exists a nonexpansive operator 
$T:{\cal H}\rightarrow{\cal H}$ such that $R=(1-\alpha)\id+\alpha T$. The conclusion follows by taking into account that 
\eqref{dyn-syst-KM-av} is equivalent to $$\left\{
\begin{array}{ll}
\dot x(t)=\alpha\lambda(t)\big(T(x(t))-x(t)\big)\\
x(0)=x_0
\end{array}\right.$$ and $\fix R=\fix T$.  
\end{proof}

In the following we investigate the convergence rate of the trajectories of the dynamical system \eqref{dyn-syst-KM}. This will be done in 
terms of the fixed point residual function $t\mapsto\|Tx(t)-x(t)\|$ and  of $t\mapsto\|\dot x(t)\|$. Notice that convergence rates for the discrete iteratively generated algorithm  
\eqref{KM-discrete} have been investigated in \cite{corman-yuan, davis-yin, liang-fadili-peyre}. 

\begin{theorem}\label{conv-KM-rate1} Let $T:{\cal H}\rightarrow {\cal H}$ be a nonexpansive mapping such that $\fix T\neq\emptyset$, 
$\lambda:[0,+\infty)\rightarrow [0,1]$ a Lebesgue measurable function and 
$x_0\in {\cal H}$. Suppose that 
$$0< \inf_{t\geq 0}\lambda(t)\leq\sup_{t\geq 0}\lambda(t)<1.$$ 
Let  $x:[0,+\infty)\rightarrow{\cal H}$ be the unique strong global solution of \eqref{dyn-syst-KM}. Then 
for all $t> 0$ we have $$\|\dot x(t)\|\leq \|T(x(t))-x(t)\|\leq\frac{d(x_0,\fix T)}{\sqrt{\ul \tau t}},$$
where $\ul\tau=\inf_{t\geq0}\lambda(t)(1-\lambda(t))>0$. 
\end{theorem}

\begin{proof} Take an arbitrary $y\in \fix T$ and $t>0$. From \eqref{ineq-x-dtx} we have for all $s\geq 0$: 

\begin{equation}\label{ineq-x-dtx-rate}\frac{d}{ds}\|x(s)-y\|^2+\lambda(s)(1-\lambda(s))\|T(x(s)-x(s))\|^2\leq 0.\end{equation}

By integrating we obtain $$\int_0^t\lambda(s)(1-\lambda(s))\|T(x(s))-x(s)\|^2 ds\leq \|x_0-y\|^2-\|x(t)-y\|^2\leq\|x_0-y\|^2.$$
We have seen in the proof of Theorem \ref{conv-KM} that $t\mapsto\frac{1}{2}\|T(x(t))-x(t)\|^2$ is decreasing, thus the 
last inequality yields  $$t\ul\tau\|T(x(t))-x(t)\|^2\leq \|x_0-y\|^2.$$
Since this inequality holds for an arbitrary $y\in\fix T$, we get for all $t\geq 0:$ 
$$\sqrt{t\ul\tau}\|T(x(t))-x(t)\|\leq d(x_0,\fix T).$$

By taking also into account \eqref{dyn-syst-KM}, the conclusion follows. 
\end{proof}

Next we show that the convergence rates of fixed point residual function $t\mapsto\|Tx(t)-x(t)\|$ and  of $t\mapsto\|\dot x(t)\|$ can be improved to $o\left(\frac{1}{\sqrt t}\right)$. 

\begin{theorem}\label{conv-KM-rate2} Let $T:{\cal H}\rightarrow {\cal H}$ be a nonexpansive mapping such that $\fix T\neq\emptyset$, 
$\lambda:[0,+\infty)\rightarrow [0,1]$ a Lebesgue measurable function and 
$x_0\in {\cal H}$. Suppose that 
$$0<\inf_{t\geq 0}\lambda(t)\leq\sup_{t\geq 0}\lambda(t)<1.$$ 
Let $x:[0,+\infty)\rightarrow{\cal H}$ be the unique strong global solution of \eqref{dyn-syst-KM}. Then for all $t\geq 0$ we have 
$$t\|\dot x(t)\|^2\leq t\|T(x(t))-x(t)\|^2\leq \frac{2}{{\ul\tau}} \int_{t/2}^t\lambda(s)(1-\lambda(s))\|T(x(s))-x(s)\|^2 ds,$$
where $\ul\tau=\inf_{t\geq0}\lambda(t)(1-\lambda(t))>0$ and $\lim_{t\rightarrow+\infty}\int_{t/2}^t\lambda(s)(1-\lambda(s))\|T(x(s))-x(s)\|^2 ds=0$.
\end{theorem}

\begin{proof} Define the function $f:[0,+\infty)\rightarrow[0,+\infty)$,
$$f(t)=\int_0^t\lambda(s)(1-\lambda(s))\|T(x(s))-x(s)\|^2 ds.$$
According to \eqref{int-l-T}  we have that $\lim_{t\rightarrow+\infty}f(t)\in\R$. 

Since $t\mapsto\frac{1}{2}\|T(x(t))-x(t)\|^2$ is decreasing (see the proof of Theorem \ref{conv-KM}), we have for all $t\geq 0:$
\begin{align*}
\|T(x(t))-x(t)\|^2\int_{t/2}^t\lambda(s)(1-\lambda(s))ds \leq & \int_{t/2}^t\lambda(s)(1-\lambda(s))\|T(x(s))-x(s)\|^2 ds\\
                                                            = & f(t)-f(t/2).
 \end{align*}
Taking into account the definition of $\ul \tau$, we easily derive 
$$\frac{{\ul\tau}}{2}t\|T(x(t))-x(t)\|^2\leq \int_{t/2}^t\lambda(s)(1-\lambda(s))\|T(x(s))-x(s)\|^2 ds,$$
and the conclusion follows by using again \eqref{dyn-syst-KM}. 
\end{proof}

The rest of the paper is dedicated to the formulation and investigation of a continuous version of the forward-backward algorithm. 
For readers convenience let us recall some standard notions and results in monotone operator theory 
which will be used in the following (see also \cite{bo-van, bauschke-book, simons}). For an arbitrary set-valued operator $A:{\cal H}\rightrightarrows {\cal H}$ we denote by 
$\gr A=\{(x,u)\in {\cal H}\times {\cal H}:u\in Ax\}$ its graph.
We use also the notation $\zer A=\{x\in{\cal{H}}:0\in Ax\}$ for the set of zeros of $A$. We say that $A$ is monotone, if $\langle x-y,u-v\rangle\geq 0$ for all $(x,u),(y,v)\in\gr A$. A monotone operator $A$ is said to be maximally monotone, if there exists no proper monotone extension of the graph of $A$ on ${\cal H}\times {\cal H}$.
The resolvent of $A$, $J_A:{\cal H} \rightrightarrows {\cal H}$, is defined by $J_A=(\id_{{\cal H}}+A)^{-1}$, where $\id_{{\cal H}} :{\cal H} \rightarrow {\cal H}, \id_{\cal H}(x) = x$ for all $x \in {\cal H}$, is the identity operator on ${\cal H}$. Moreover, if $A$ is maximally monotone, then $J_A:{\cal H} \rightarrow {\cal H}$ is single-valued and maximally monotone
(see \cite[Proposition 23.7 and Corollary 23.10]{bauschke-book}). For an arbitrary $\gamma>0$ we have (see \cite[Proposition 23.2]{bauschke-book})
\begin{equation}p\in J_{\gamma A}x \ \mbox{if and only if} \ (p,\gamma^{-1}(x-p))\in\gr A.\end{equation}

The operator $A$ is said to be uniformly monotone if there exists an increasing function
$\phi_A : [0,+\infty) \rightarrow [0,+\infty]$ that vanishes only at $0$, and
$\langle x-y,u-v \rangle \geq \phi_A \left( \| x-y \|\right)$ for every $(x,u)\in\gr A$ and $(y,v) \in \gr A$. A well-known 
class of operators fulfilling this property is the one of the strongly monotone operators.
Let $\gamma>0$ be arbitrary. We say that $A$ is $\gamma$-strongly monotone, 
if $\langle x-y,u-v\rangle\geq \gamma\|x-y\|^2$ for all $(x,u),(y,v)\in\gr A$. We consider also the class of cocoercive operators: 
$B:{\cal H}\rightarrow {\cal H}$ is $\gamma$-cocoercive, 
if $\langle x-y,Bx-By\rangle\geq \gamma\|Bx-By\|^2$ for all $x,y\in {\cal H}$. 

\begin{theorem}\label{fb-dyn} Let $A:{\cal H}\rightrightarrows {\cal H}$ be a maximally monotone operator, $\beta>0$ and 
$B:{\cal H}\rightarrow {\cal H}$ be $\beta$-cocoercive such that $\zer(A+B)\neq\emptyset$. Let $\gamma\in(0,2\beta)$ and set 
$\delta=\min\{1,\beta/\gamma\}+1/2$. Let $\lambda:[0,+\infty)\rightarrow [0,\delta]$ be a Lebesgue measurable function and 
$x_0\in {\cal H}$. Suppose that one if the following conditions is fulfilled: 
$$\int_0^{+\infty}\lambda(t)(\delta-\lambda(t))dt=+\infty \ \mbox{or} \ \inf_{t\geq 0}\lambda(t)>0.$$ Let 
$x:[0,+\infty)\rightarrow{\cal H}$ be the unique strong global solution of 

\begin{equation}\label{dyn-syst-fb}\left\{
\begin{array}{ll}
\dot x(t)=\lambda(t)\left[J_{\gamma A}\Big(x(t)-\gamma B(x(t))\Big)-x(t)\right]\\
x(0)=x_0.
\end{array}\right.\end{equation}
 Then the following statements are true: 

(i) the trajectory $x$ is bounded and $\int_0^{+\infty}\|\dot x(t)\|^2dt<+\infty$;  

(ii) $\lim_{t\rightarrow+\infty}\left[J_{\gamma A}\Big(x(t)-\gamma B(x(t))\Big)-x(t)\right]=0$; 

(iii) $\lim_{t\rightarrow+\infty}\dot x(t)=0$;

(iv) $x(t)$ converges weakly to a point in $\zer(A+B)$, as $t\rightarrow+\infty$.

\noindent Suppose that $\inf_{t\geq 0}\lambda(t)>0$. Then the following hold: 

(v) if $y\in\zer(A+B)$, then $\lim_{t\rightarrow+\infty}B(x(t))=By$ and $B$ is constant on $\zer(A+B)$; 

(vi) if $A$ or $B$ is uniformly monotone, then $x(t)$ converges strongly to the unique point in $\zer(A+B)$, as $t\rightarrow+\infty$.
\end{theorem}

\begin{proof} It is immediate that the dynamical system \eqref{dyn-syst-fb} can be written in the form 

\begin{equation}\label{dyn-syst-KM-fb}\left\{
\begin{array}{ll}
\dot x(t)=\lambda(t)\big(T(x(t))-x(t)\big)\\
x(0)=x_0,
\end{array}\right.\end{equation}
where $T=J_{\gamma A}\circ(\id -\gamma B).$ According to \cite[Corollary 23.8 and Remark 4.24(iii)]{bauschke-book}, $J_{\gamma A}$ is $1/2$-cocoercive. 
Moreover, by \cite[Proposition 4.33]{bauschke-book}, $\id -\gamma B$ is $\gamma/(2\beta)$-averaged. Combining this with 
\cite[Proposition 4.32]{bauschke-book}, we derive that $T$ is $1/\delta$-averaged. The statements (i)-(iv) follow now from 
Corollary \ref{conv-KM-av} by noticing that $\fix T=\zer(A+B)$, see \cite[Proposition 25.1(iv)]{bauschke-book}. 

We suppose in the following that $\inf_{t\geq 0}\lambda(t)>0$. 

(v) The fact that $B$ is constant on $\zer(A+B)$ follows from the cocoercivity of $B$ and the monotonicity of $A$. A proof of this statement when $A$ is 
the subdifferential of a proper, convex and lower semicontinuous function is given in \cite[Lema 1.7]{abbas-att-arx14}. 

We use the following inequality:
\begin{equation}\label{ineq-T-fb}\|Tx-Ty\|^2\leq\|x-y\|^2-\gamma(2\beta-\gamma)\|Bx-By\|^2 \ \forall (x,y)\in {\cal H}\times {\cal H},\end{equation}
which follows from the nonexpansiveness property of the resolvent and the cocoercivity of $B$: 
\begin{align*}
\|Tx-Ty\|^2\leq & \|x-y-\gamma(Bx-By)\|^2=\|x-y\|^2-2\gamma\<x-y,Bx-By\>+\gamma^2\|Bx-By\|^2\\
\leq & \|x-y\|^2-\gamma(2\beta-\gamma)\|Bx-By\|^2.
\end{align*}
Take an arbitrary $y\in\zer(A+B)=\fix T$. From the first part of the proof of Theorem \ref{conv-KM} and \eqref{ineq-T-fb} we get for all $t\geq 0$ 
\begin{align*}
& \frac{d}{dt}\|x(t)-y\|^2+\lambda(t)(1-\lambda(t))\|T(x(t)-x(t))\|^2+\|\dot x(t)\|^2 \\
& = \lambda(t)\|T(x(t))-y\|^2-\lambda(t)\|x(t)-y\|^2\leq-\gamma(2\beta-\gamma)\lambda(t)\|B(x(t))-By\|^2.
\end{align*}

Taking into account that $\inf_{t\geq 0}\lambda(t)>0$ and $0<\gamma<2\beta$, by integrating the above inequality 
we obtain $$\int_0^{+\infty}\|B(x(t))-By\|^2dt<+\infty.$$

Since $B$ is $1/\beta$-Lipschitz (this follows from the $\beta$-cocoercivity of $B$ by applying the 
Cauchy-Schwarz inequality) and $t\mapsto\|\dot x(t)\|\in L^2([0,+\infty))$, from Remark \ref{rem-abs-cont}(b) we derive that 
$t\mapsto\frac{d}{dt} B(x(t))\in L^2([0,+\infty),{\cal H})$. From the Cauchy-Schwarz inequality we obtain for all $t\geq 0$
$$\frac{d}{dt}\left(\left\|B(x(t))-By\right\|^2\right)=2\<\frac{d}{dt}B(x(t)),B(x(t))-By\>\leq \left\|\frac{d}{dt}B(x(t))\right\|^2+\|B(x(t))-By\|^2.$$
Combining these considerations with Lemma \ref{fejer-cont2}, we conclude that $B(x(t))$ converges strongly to $By$, as $t\rightarrow+\infty$. 

(vi) Suppose that $A$ is uniformly monotone and let $y$ be the unique point in $\zer(A+B)$. According to \eqref{dyn-syst-fb} and the 
definition of the resolvent, we have $$-B(x(t))-\frac{1}{\gamma\lambda(t)}\dot x(t)\in A\left(\frac{1}{\lambda(t)}\dot x(t)+x(t)\right) \ \forall t\geq 0.$$
From $-By\in Ay$ we get for all $t \geq 0$  the inequality
$$\phi_A\left(\left\|\frac{1}{\lambda(t)}\dot x(t)+x(t)-y\right\|\right)\leq \<\frac{1}{\lambda(t)}\dot x(t)+x(t)-y,-B(x(t))-\frac{1}{\gamma\lambda(t)}\dot x(t)+By\>,$$
where $\phi_A:[0,+\infty)\rightarrow[0,+\infty]$ is increasing and vanishes only at $0$. 

The monotonicity of $B$ implies 
\begin{align*}
& \phi_A\left(\left\|\frac{1}{\lambda(t)}\dot x(t)+x(t)-y\right\|\right)\leq -\frac{1}{\gamma\lambda^2(t)}\|\dot x(t)\|^2 +\frac{1}{\lambda(t)}\<\dot x(t),-B(x(t))+By\> +\\
& \<x(t)-y,-B(x(t))+By\>-\frac{1}{\gamma\lambda(t)}\<\dot x(t),x(t)-y\> \\
\leq & -\frac{1}{\gamma\lambda^2(t)}\|\dot x(t)\|^2 +\frac{1}{\lambda(t)}\<\dot x(t),-B(x(t))+By\>-\frac{1}{\gamma\lambda(t)}\<\dot x(t),x(t)-y\> \forall t \geq 0.
\end{align*} 

The last inequality implies, by taking into consideration (iii), (iv) and (v), that
$$\lim_{t\rightarrow+ \infty}\phi_A\left(\left\|\frac{1}{\lambda(t)}\dot x(t)+x(t)-y\right\|\right)=0.$$ The properties of the function 
$\phi_A$ allow to conclude that $\frac{1}{\lambda(t)}\dot x(t)+x(t)-y$ converges strongly to $0$, as $t\rightarrow + \infty$, hence from  (iii) we obtain the conclusion. 

Finally, suppose that $B$ is uniformly monotone, with corresponding function $\phi_B:[0,+\infty) \rightarrow [0,+\infty]$, which is increasing  and vanishes only at $0$. The conclusion follows by taking  in the inequality
$$\<x(t)-y,B(x(t))-By\>\geq \phi_B(\|x(t)-y\|)$$ 
the limit as $t\rightarrow + \infty$ and by using (i) and (v). 
\end{proof}

\begin{remark} Let us mention that in case $A=\partial\Phi$, where $\Phi:{\cal H}\rightarrow\R\cup\{+\infty\}$ is a proper, convex 
and lower semicontinuous function defined on a real Hilbert space ${\cal H}$, and for $\lambda(t)=1$ for all $t\geq 0$, the dynamical system 
\eqref{dyn-syst-fb} becomes \eqref{syst-abb-att}, which has been studied in \cite{abbas-att-arx14}. Notice that the weak convergence 
of \eqref{syst-abb-att} is obtained in \cite[Theorem 4.2]{abbas-att-arx14} for a constant step-size $\gamma\in(0,4\beta)$.  
\end{remark}

\begin{remark} The explicit discretization of \eqref{dyn-syst-fb} with respect to the time variable $t$, with step size $h_n>0$ and initial point $x_0$, yields the following iterative scheme: 
$$\frac{x_{n+1}-x_n}{h_n}=\lambda_n\left[J_{\gamma A}\Big(x_n-\gamma Bx_n\Big)-x_n\right] \ \forall n \geq 0.$$

For $h_n=1$ this becomes
\begin{equation}\label{fb-discrete}x_{n+1}=x_n+\lambda_n\left[J_{\gamma A}\Big(x_n-\gamma Bx_n\Big)-x_n\right] \ \forall n \geq 0,\end{equation}
which is the classical forward-backward algorithm for finding the set of zeros of $A+B$  
(see \cite[Theorem 25.8]{bauschke-book}). Let us mention that the convergence of \eqref{fb-discrete} is guaranteed under the condition 
$\sum_{n \in \N} \lambda_n(\delta-\lambda_n)=+\infty$.  
\end{remark}

\begin{remark} As mentioned in the introduction, the Douglas-Rachford algorithm for finding the set of zeros 
of the sum of two maximally monotone operators follows from the discrete version of the Krasnosel'ski\u{\i}--Mann numerical scheme, 
see \cite{bauschke-book}. Following the approach presented above, one can formulate a dynamical system of Douglas-Rachford-type, the existence  and weak convergence of the trajectories being a consequence of the main results presented here.  The same can be done for other iterative schemes which have their origins in the discrete Krasnosel'ski\u{\i}--Mann algorithm, like are the generalized forward-backward splitting algorithm in \cite{rfp} and the forward-Douglas-Rachford splitting algorithm in \cite{bra}.
\end{remark}

\section{An alternative approach relying on time rescaling arguments}\label{sec4}

The content of this section has as starting point a comment made by H. Attouch on a preliminary version of this manuscript. We will show, by using time rescaling arguments, 
that the convergence behavior of the dynamical system \eqref{dyn-syst-KM} can be derived from the one of an autonomous dynamical system governed by a cocoercive operator. 
Let us recall first the following classical result, which can be deduced for example from \cite[Theorem 3.1]{abbas-att-arx14} by taking $\Phi=0$ as well as from Theorem \ref{fb-dyn} by choosing $Ax=0$ for 
all $x\in\cal{H}$ and $\lambda(t)=1$ for all $t\geq 0$. 

\begin{theorem}\label{dyn-coc} Let $B:\cal{H}\rightarrow\cal{H}$ be a cocoercive operator such that $\zer B\neq\emptyset$ and  
$w_0\in {\cal H}$. Let $w:[0,+\infty)\rightarrow{\cal H}$ be the unique strong global solution of 
the dynamical system \begin{equation}\label{dyn-syst-coc}\left\{
\begin{array}{ll}
\dot w(t)+B(w(t))=0\\
w(0)=w_0.
\end{array}\right.\end{equation} Then the following statements are true: 

(a) the trajectory $w$ is bounded and $\int_0^{+\infty}\|\dot w(t)\|^2dt<+\infty$;  

(b) $w(t)$ converges weakly to a point in $\zer B$, as $t\rightarrow+\infty$; 

(c) $B(w(t))$ converges strongly to $0$, as $t\rightarrow+\infty$.  
\end{theorem}

Let us consider again the dynamical system \eqref{dyn-syst-KM}, written as

\begin{equation*}\label{dyn-syst-KM2}\left\{
\begin{array}{ll}
\dot x(t)+\lambda(t)(\id-T)(x(t))=0\\
x(0)=x_0.
\end{array}\right.\end{equation*}

We recall that $T$ is nonexpansive such that $\fix T\neq\emptyset$ and $\lambda:[0,\infty)\rightarrow[0,1]$ is Lebesgue measurable. 
By using a time rescaling argument as in \cite[Lemma 4.1]{att-cza-10}, we can prove a connection between the dynamical system 
\eqref{dyn-syst-KM} and the system \begin{equation}\label{dyn-syst-coc2}\left\{\begin{array}{ll}
\dot w(t)+(\id -T)(w(t))=0\\
w(0)=x_0.
\end{array}\right.\end{equation}

In the following we suppose that \begin{equation}\label{cond-l}\int_0^{+\infty}\lambda(t)dt=+\infty.\end{equation}
Notice that the considerations which we make in the following remain valid also when one requires for the function $\lambda$ an arbitrary positive upper bound. However, we 
choose as upper bound $1$ in order to remain in the setting presented in the previous section. 

Suppose that we have a solution $w$ of \eqref{dyn-syst-coc2}. By defining the function $T_1:[0,+\infty)\rightarrow[0,+\infty)$, 
$T_1(t)=\int_0^t\lambda(s)ds$, one can easily see that $w\circ T_1$ is a solution of \eqref{dyn-syst-KM}.

Conversely, if $x$ is a solution of \eqref{dyn-syst-KM}, then $x\circ T_2$ is a solution of \eqref{dyn-syst-coc2}, where 
$T_2:[0,+\infty)\rightarrow[0,+\infty)$ is defined for every $t \geq 0$ implicitly as $\int_0^{T_2(t)}\lambda(s)ds=t$ (this is possible due to the 
properties of the the function $\lambda$). 

In the arguments above we used that \begin{equation}\label{t1}T_1'(t)=\lambda(t) \ \forall t\geq 0\end{equation}
and \begin{equation}\label{t2}T_2'(t)\lambda(T_2(t))=1 \ \forall t\geq 0.\end{equation}

Further, since $B:=\id-T$ is $1/2$-cocoercive (this follows from the nonexpansiveness of $T$), 
for the dynamical system \eqref{dyn-syst-coc2} one can apply the convergence results presented in Theorem \ref{dyn-coc}. 
We would also like to notice that the existence of a strong global solution of \eqref{dyn-syst-KM} follows from the corresponding
result for \eqref{dyn-syst-coc2}, while for the uniqueness property we have to make use of the considerations at the end of Section \ref{sec2}.

In the following we deduce the convergence statements of Theorem \ref{conv-KM} from the ones of Theorem \ref{dyn-coc} by using the time rescaling arguments presented above. 

Let $x$ be the unique strong global solution of \eqref{dyn-syst-KM}. Due to the uniqueness of the solutions of \eqref{dyn-syst-KM} and 
\eqref{dyn-syst-coc2}, we have $x=w\circ T_1$, where $w$ is the unique strong global solution of \eqref{dyn-syst-coc2}.

(i) From Theorem \ref{dyn-coc}(a) we know that $w$ is bounded, hence $x$ is bounded, too. We have 
\begin{align*}
\int_0^{+\infty}\|\dot x(s)\|^2ds & =\lim_{t\rightarrow+\infty}\int_0^t \|w'(T_1(s))\|^2(\lambda(s))^2ds\leq 
\lim_{t\rightarrow+\infty}\int_0^t \|w'(T_1(s))\|^2\lambda(s)ds\\
& =\lim_{t\rightarrow+\infty}\int_0^{T_1(t)} \|w'(u)\|^2du<+\infty,
\end{align*}
where we used Theorem \ref{dyn-coc}(a) and the change of variables $T_1(s)=u$. 

(ii) This statement follows from Theorem \ref{dyn-coc}(c).

(iii) Is a direct consequence of the boundedness of $\lambda$, (ii) and of the way the dynamic is defined. 

(iv) From Theorem \ref{dyn-coc}(b) it follows that $x(t)=w(T_1(t))$ converges weakly to a point in $\zer B=\fix T$ as $t \rightarrow +\infty$. 

\begin{remark} 
In the light of the above considerations it follows that the conclusion of Theorem \ref{conv-KM} remains valid also when assuming that $\int_0^{+\infty}\lambda(t)dt=+\infty$, which is a weaker condition than asking that $\int_0^{+\infty}\lambda(t)(1-\lambda(t))dt=+\infty$ or $\inf_{t \geq 0} \lambda(t) >0$. A similar statement applies to Theorem \ref{fb-dyn}, too. Notice also that the assumption that $\lambda$ takes values in $[0,1]$, being strictly bounded away from the endpoints of this interval, was essential, in combination to the considerations made in the proof of Theorem \ref{conv-KM}, for deriving convergence rates for the trajectories of \eqref{dyn-syst-KM}. Finally, let us mention that, as pointed out in Remark \ref{contdisc},  the assumption $\int_0^{+\infty}\lambda(t)(1-\lambda(t))dt=+\infty$ has a natural counterpart in the discrete case which guarantees convergence for the sequence of generated iterates, while this is not the case for the other two
conditions on $\lambda$ considered in this paper.
\end{remark}

\noindent{\bf Acknowledgements.} We are thankful to H. Attouch for bringing into our attention the time rescaling arguments 
which led to the results presented in the last section.\\

\end{document}